\documentclass{era}
\usepackage{amsmath}
\usepackage{amsmath, amsthm, amsfonts, amssymb}
  \usepackage{paralist}
  \usepackage{graphics} 
  \usepackage{epsfig} 
\usepackage{graphicx}  \usepackage{epstopdf}
 \usepackage[colorlinks=true]{hyperref}
\hypersetup{urlcolor=blue, citecolor=red}

\def\currentvolume{28}
 \def\currentissue{2}
  \def\currentyear{2020}
   \def\currentmonth{June}
    \def\ppages{X--XX}
     \def\DOI{10.3934/era.2020030}

\newtheorem{theorem}{Theorem}[section]
\newtheorem{corollary}{Corollary}

\newtheorem{lemma}[theorem]{Lemma}
\newtheorem{proposition}{Proposition}

\theoremstyle{definition}
\newtheorem{definition}[theorem]{Definition}
\newtheorem{remark}{Remark}
\newtheorem{example}{Example}

\title[Riemann-Liouville derivative of  integrable distributions] 
      {Riemann-Liouville derivative over the space of integrable distributions }

\author[Mar\'ia Guadalupe Morales, Zuzana Do\v{s}l\'a and Francisco J. Mendoza]{}


 \email{maciasm@math.muni.cz}
 \email{dosla@math.muni.cz}
 \email{jmendoza@fcfm.buap.mx}

\thanks{The second author is supported by Grant GA 17-03224S of the Czech Science Foundation. The third author acknowledges partial support by SNI-CONACYT and VIEP-BUAP}

\thanks{$^*$ Corresponding author: Mar\'ia Guadalupe Morales.}

\begin{document}
\maketitle

\centerline{\scshape Mar\'ia Guadalupe Morales$^*$ and Zuzana Do\v{s}l\'a}
\medskip
{\footnotesize
 \centerline{Department of Mathematics and Statistics}
   \centerline{Faculty of Science, Masaryk University}
   \centerline{Kotl\'a\v{r}sk\'a 2, 611 37 Brno, Czech Republic}
} 

\medskip

\centerline{\scshape Francisco J. Mendoza}
\medskip
{\footnotesize
 \centerline{Facultad de Ciencias F\'isico-Matem\'aticas}
   \centerline{ Benem\'erita Universidad Aut\'onoma de Puebla}
   \centerline{Av. San Claudio y 18 Sur S/N, Puebla, Puebla, 72570, M\'exico}
}
\bigskip

 \centerline{(Communicated by Abdon Atangana)}

\begin{abstract}
In this paper, we generalize the  Riemann-Liouville differential and integral operators on the space of  Henstock-Kurzweil integrable distributions, $D_{HK}$. We obtain new fundamental properties of the fractional derivatives and integrals, a general version of the fundamental theorem of fractional calculus, semigroup property for the Riemann-Liouville integral operators and relations between the Riemann-Liouville integral and differential operators. Also, we achieve a generalized characterization of the solution for the Abel integral equation. Finally, we show relations for the Fourier transform of fractional derivative and integral. These results are based on the properties of the distributional Henstock-Kurzweil integral and convolution.
\end{abstract}

\section{Introduction}

Fractional calculus is devoted to studying the different
possibilities of defining the differential operator of arbitrary
order, properties, possible relations with the integral operator,
and applications. There exist several classical
definitions for fractional derivatives, for example, Caputo,
Riemann-Liouville, Marchaud, Hadamard and the Weyl derivatives,
among others. In recent years with the intention
of solving some problems that the classical fractional derivatives
do not achieve and as a consequence of those, new fractional derivatives have been defined combining power law,
exponential decay and Mittag-Leffler kernel; among them those of
Liouville-Caputo, Atanga-Caputo, Atanga-G\'{o}mez and
Atanga-Baleanu derivatives, see \cite{At1}, \cite{At2}, \cite{At4}, \cite{At5} and
\cite{Go2}.

This branch of mathematics is one of the most powerful
 modeling tools applied in many fields of
science, physics, chemistry, biology, engineering, see e.g. \cite{At1}, \cite{At2}%
, \cite{At4},  \cite{A3}, \cite{A9}, \cite{Go1},  \cite{A15},
\cite{Mo1} and \cite{Ta1}.
In general, fractional calculus has been developed in the context
of the Lebesgue integral, see e.g. \cite{A6}, \cite{A12},
\cite{A14}, \cite{A20} and \cite{A27}.

In this article, we use  more general integrals than the Lebesgue integral,
expanding the space of functions where it is possible to apply
fractional calculus. Although our study is theoretical, we primarily pursue an aim partially similar to those developed in \cite{At2}, \cite{At3}, \cite{At6}, \cite%
{Go2} and \cite{Ye1}. For example, the use of generalized integration helps us to recover the initial function, a fact searched in the above references.

The basic idea behind fractional calculus is the fundamental theorem of calculus (FTC).
First, we denote by $J_a$ the operator that maps a integrable function $f$ (in some sense, for example Riemann integrable in the compact interval $[a,b]$), onto its primitive centered at $a$. We denote by $D$ the operator that maps a differentiable function onto its derivative,
$Df:=f'(x).$ For $n\in\mathbb{N}$, $J_a^n$ and  $D^n$ denote the  $n$-fold iterates of $ J_a$ and $D$, respectively. Therefore, using this notation the FTC reads as $$DJ_af(x)=f(x)\ \ \text{a.e. on }\ \ [a,b],$$
and this implies that
\begin{equation}
D^nJ_a^nf(x)=f(x) \ \ \text{a.e. on }\ \ [a,b],\label{3}
\end{equation}  for $n\in\mathbb{N}$.
Moreover, it is possible to obtain by induction that for any $n\in \mathbb{N}$ and $f$  Riemann-integrable function, $J_a^n$  is given as
\begin{equation*}\label{defintn}
J_a^nf(x)=\frac{1}{(n-1)!}\int_a^x (x-t)^{n-1}f(t)\,dt,
\end{equation*}where  $a\leq x\leq b$.  Using the Riemann-Liuoville fractional derivative the equality (\ref{3}) holds for any $n>0$, even non-integer, see e.g. \cite{A6}, \cite{A12} and \cite{A20}. This means that we have a version of the FTC in a fractional sense.
Besides, according to this fractional derivative, an arbitrary function does not need to
be continuous nor be differentiable in the usual sense. Thus, this definition is an excellent tool for the description of memory and hereditary properties of several materials and processes, \cite{A12}, \cite{A20} and \cite{A27}.

On the other hand, the integration theory continues developing. For example, in the last century R. Henstock and J. Kurzweil introduced a generalized integral, it is known as Henstock-Kurzweil integral, see e.g. \cite{A2} and \cite{A21}. Later, E. Talvila obtained fundamental properties about the Fourier transform using the Henstock-Kurzweil integral, see \cite{A23}.
Moreover, the Henstock-Kurzweil integral can be generalized in various ways, for example, one can consider Henstock-Kurzweil-Stieljtes type integrals, see \cite{Mo}. Another possible direction is the distribution theory, in \cite{A22} and \cite{A24} E. Talvila extended this integral in a distributional sense and achieved new properties for the convolution in a generalized sense. Thus, with the introduction of new
integration theories, the possibility to extend fundamental results arises.  The following relations are well known,
\begin{equation}\label{subsets} L^1[a,b]\subsetneq HK[a,b]\subsetneq \widehat{HK[a,b]} \simeq D_{HK},
\end{equation}
where $[a,b]$ is any compact interval, the space of Lebesgue integrable functions is denoted by $L^1[a,b]$ and the space of Henstock-Kurzweil integrable function  by $ HK[a,b]$. $\widehat{HK[a,b]} $ denotes the completion of $ HK[a,b]$ with  respect to the Alexiewicz norm, and  $D_{HK}$  the space of all distributions each of which is the distributional derivative of a continuous function. The last relation in (\ref{subsets}) means that the completion of $HK[a,b]$ is isometrically isomorphic to $D_{HK}$, see \cite{A4} and \cite{A5}.
The subsets of integrable functions are strictly contained, and the values of the integrals coincide; see \cite{A13} , \cite{A16} and \cite{A17}. The Riemann, Lebesgue, Henstock-Kurzweil, Perron, Denjoy, and improper integrals are special cases contained in the distributional Henstock-Kurzweil integral. Moreover, it is valid on unbounded intervals, see \cite{A24}.

 Motivated by the suitability and applicability of the Riemann-Liouville fractional derivative, and the generality offered by the distributional integral,
 we define the Riemann-Liouville fractional integral and differential operators in the context of the distributional Henstock-Kurzweil integral. Also, we extend fundamental properties (see e.g., Theorem \ref{coro1} and Theorem \ref{semigroup}) and obtain new relations between the fractional integral and differential operators, see Theorem \ref{prop2}. In particular, we prove that the fractional differential operator inverts the fractional integral in a distributional sense, see Theorem \ref{teo6}.  Finally, we show some applications of the fractional derivative, for example, a general characterization of the solution for the Abel integral equation and new properties of the Fourier transform for the fractional integral and derivative.

\section{Preliminaries}

Following the notation from \cite{A6}, we introduce the Riemann-Liouville integral and differential operators.

First, we recall that function $\Gamma:(0,\infty)\rightarrow\mathbb{R},$  defined by
\begin{equation}\label{gamma} \Gamma(x):=\int_0^\infty t^{x-1}e^{-t}dt,
\end{equation}
is  Euler's Gamma function. 
Let $n$ be a positive number. We define $m=\lceil n \rceil$ as  the least integer greater than or equal to  $n$.

\begin{definition}Let $n\in\mathbb{R}^+\cup \{0\}$. The operator $J_a^n$ is defined in $L^1[a,b]$ by
\begin{equation}
J_a^nf(x):=\frac{1}{\Gamma(n)}\int_a^x (x-t)^{n-1}f(t)\,dt
\end{equation}
for $a\leq x\leq b$, is called the \textit{Riemann-Liouville fractional integral operator of order $n$}. For $n=0$, we set $J_a^0:=I$, the identity operator.
\end{definition}
Note that for any $f\in L^1[a,b]$ and $n\in\mathbb{R}^+\cup \{0\}$,  the operator $J_a^nf(x)$ exists for almost every $x\in[a,b]$ and is also an element of $L^1[a,b]$,  it means for any $n\in\mathbb{R}^+$
$$J_a^n(L^1[a,b])\subset L^1[a,b],$$see e.g.  \cite[Theorem 2.1]{A6}.
\begin{definition}\label{difoperator} Let $n\in\mathbb{R}^+\cup \{0\}$ and $m=\lceil n \rceil$. The operator $D_a^n$ is defined by
\begin{equation}
D_a^nf:=D^mJ_a^{m-n}f
,
\end{equation}
when $D_a^nf$ exists a.e. on $[a, b]$ is called the \textit{Riemann-Liouville fractional differential operator of order $n$}, where $D^{m}$ denotes the $m-$folds iterates of the derivative.  For $n=0$, we set $D_a^0:=I$, the identity operator.
\end{definition}
Note than $D_a^nf$ might not exist. However, there are sufficient conditions for the existence of fractional
derivatives provided e.g. in \cite{A6}, \cite{A12} and \cite{A20}. In \cite[Theorem 2.14]{A6} it is shown that the Riemann-Liouville fractional differential operator inverts the Riemann-Liouville fractional integral operator over $L^1[a,b]$, where $[a,b]$ is any compact interval. This means that the FTC holds over $L^1[a,b]$ considering the Riemann-Liouville fractional derivative. This result is known as the fundamental theorem of fractional calculus.

Now we introduce the definition of the distributional Henstock-Kurzweil integral. Recall, the Lebesgue integral is characterized  in terms of absolutely continuous functions, $AC$. In the case of Henstock-Kurzweil integral, there is an  analogous characterization   in terms of  generalized absolutely continuous functions in the restricted sense $ACG_*$. This means,  $F\in ACG_*$ if and only if there exists $f\in HK[a,b]$ such that $F(x)=\int_a^xf+F(a)$, hence  $F'=f$ a.e., see \cite{A10}. However, if $F$ is a continuous function, then the generalized function  and the distributional derivative are needed because there are  continuous functions that are differentiable nowhere.

Let $(a,b)$ be an bounded  open interval in $\mathbb{R}$, we define
\[\mathcal{D}(a,b):=\{\phi:(a,b)\rightarrow \mathbb{R}\ | \ \phi\in C^\infty \text{ and }  \phi \text{ has a compact support in } (a,b) \}.\]
Moreover, it is said  that a sequence $(\phi_n)\subset \mathcal{D}(a,b)$ converges to $\phi\in \mathcal{D}(a,b)$ if there is a compact set $K\subset (a,b)$ such that all $\phi_n$ have support in $K$ and for each integer $m\geq 0$, the sequence of derivatives $(\phi_n^{(m)})$ converges to $\phi^{(m)}$ uniformly on $K$, see e.g. \cite{A11}.

The dual space of $\mathcal{D}(a,b)$ is denoted by $\mathcal{D}'(a,b)$ and is the space of  continuous linear functionals on $\mathcal{D}(a,b)$. This refers  to the distributions on $(a,b).$ Let
$$C_0:=\{F\in C[a,b]:F(a)=0\}.$$
It is well known that $C_0$ is a Banach space with the uniform norm,  $||F||_\infty:=\sup_{t\in[a,b]}|F(t)|$.

We will follow the notation from \cite{A25} to introduce the distributional Henstock-Kurzweil integral.
\begin{definition}\label{defdistri} A distribution $f\in \mathcal{D}'(a,b) $ is said to be a Henstock-Kurzweil integrable distribution on  $[a,b]$ if there exists a continuous function $F\in C_0$ such that $F'=f$ (the distributional derivative of $F$ is $f$). In other words, $F$ is the primitive of $f$. The distributional Henstock-Kurzweil integral of $f$ on $[a,b]$ is denoted by $$\int_a^bf(t)\,dt:=F(b)-F(a).$$
\end{definition}
We set $DF:=F'$ in a distributional sense. The space of all the  Henstock-Kurzweil integrable  distributions on $[a,b]$ is denoted  by $D_{HK}$. For $f\in D_{HK}$, we define the Alexiewicz norm in $ D_{HK}$ as
$$||f||_A:=||F||_\infty,$$ where $DF=f$. In particular, if $f\in HK[a,b]$, then $||f||_A:=\sup_{x\in[a,b]}|\int_a^xf|.$ Now let us consider  $f\in D_{HK}$ and $(f_k)\subset HK[a,b]$ such that $||f_k-f||_A\rightarrow 0$, as $k\rightarrow \infty$. Let us denote for each $k\in\mathbb{N}$, $F_k$ as the primitive of $f_k$ ($F_k'(x)=f_k(x)$ a.e.). Since $(f_k)\subset HK[a,b]$, then  $\langle f_k,\phi \rangle=\int_a^bf_k(t)\phi(t)\,dt$ for all $\phi\in \mathcal{D}(a,b)$ and
$$\langle f_k,\phi \rangle       : =\int_a^bf_k(t)\phi(t)\,dt=-\int_a^bF_k\phi'.$$
On the other hand, since $||f-f_k||_{A}\rightarrow 0$, as $k\rightarrow \infty$,  $(F_k)$ is a Cauchy sequence in $C[a,b]$. Therefore, exists $F\in C[a,b]$ such that $F_k(x)\rightarrow F(x)$ and for every $\phi\in\mathcal{D}(a,b)$
$$\lim_{k\rightarrow\infty}\langle f_k,\phi \rangle        =-\lim_{k\rightarrow\infty}\int_a^bF_k\phi'=-\int_a^bF\phi'=-\langle F,\phi' \rangle=\langle F',\phi \rangle.   $$
Thus,  in the sense of distributions $(f_k)$ converges (weakly) to $F'$. Now, by H\"older inequality, \cite[Theorem 7]{A24},
$$\left|\int_a^b (f_k-f)\phi\right|\leq 2||f_k-f||_A||\phi||_{BV}\;\; \text{for each}\;\; \phi\in\mathcal{D}(a,b).$$
Hence $f_k\rightarrow f$, in the distributions sense. Therefore $f=F'$.


Note that it  does not depend on the Cauchy sequence because  the set of conti\-{nuous} functions with uniform norm  is a Banach space.
Moreover, if $f\in D_{HK}$, then $f$ has many primitives in $C[a,b]$, all differing  by a constant. Nevertheless, $f$ has exactly one primitive in $C_0$, see \cite[Theorem 6, $ii$)]{A5}.

By \cite[Theorem 2, Theorem 3]{A24},  $D_{HK}$ is a Banach space and is a separable space with respect to the Alexiewicz norm, respectively. 
On the other hand, in \cite{A4} and \cite{A5} it is shown that the completion of the  Henstock-Kurzweil integrable functions space, $\widehat{HK[a,b]}$,  is isomorphic to $D_{HK}$.
Furthermore, in \cite{A5}, \cite{A24} and \cite{A26} the following result is proved.
\begin{theorem}\label{isomorphic} $D_{HK}$ is isomorphic to the space $C_0$.
\end{theorem}

Another important fact is that the Banach dual of $HK[a,b]$ is isomorphic to the space $BV[a,b]$ of all functions of bounded variations on $[a,b]$. Moreover, $HK[a,b]^*=D_{HK}^*=BV[a,b]$, see \cite{A1}. To consult the formal definitions, see \cite{A10}. Also there exists a version of FTC in the Henstock-Kurzweil distributional sense.
\begin{theorem}{\rm{(\cite[Theorem 4]{A24} Fundamental theorem of calculus)}}\label{FTCdis}
\begin{itemize}
\item[(i)] Let $f\in D_{HK}$ and $F(x):=\int_a^x f.$ Then $F\in C_0$ and $DF=f$.
\item[(ii)] Let $F\in C[a,b].$ Then $\int_a^x DF=F(x)-F(a)$ for all $x\in [a,b]$.
\end{itemize}
\end{theorem}

\begin{definition} Let $f$ and  $(f_k)$ in $ D_{HK}$.
\begin{itemize}
\item[(i)] $(f_k)$ converges in Alexiewicz norm to $f$ if $||f_k-f||_A\rightarrow 0$ as $k\rightarrow\infty$.
\item[(ii)] $(f_k)$ converges weakly to $f$ in $\mathcal{D}(a,b)$ if $\langle f_k-f,\phi\rangle=\int_a^b(f_k-f)\phi\rightarrow 0 $ ($k\rightarrow\infty$) for each $\phi\in\mathcal{D}(a,b).$
\item[(iii)]  $(f_k)$ converges weakly to $f$ in $BV$ if $\langle f_k-f,g\rangle=\int_a^b(f_k-f)g\rightarrow 0 $ ($k\rightarrow\infty$) for each $g\in BV[a,b]$.

\end{itemize}
\end{definition}

In \cite{A24} is proved the following result.
\begin{theorem}\label{convergence} We have
\begin{itemize}
\item [(i)] Convergence in Alexiewicz norm implies weak convergence in $\mathcal{D}(a,b)$ and $BV[a,b]$.
\item[(ii)] Weak convergence in $BV[a,b]$ implies weak convergence in $\mathcal{D}(a,b)$.
\item[(iii)] Nevertheless, weak convergence in $\mathcal{D}(a,b)$ does not imply weak convergence in $BV[a,b]$ or weak convergence in $BV[a,b]$ does not imply convergence in Alexiewicz norm.

\end{itemize}

\end{theorem}

Consider
 \[\mathcal{D}:=\{\phi:\mathbb{R}\rightarrow\mathbb{R}\ |\ \phi\in C^\infty \text { and }\phi \text{ has a compact support on } \mathbb{R}\}.\]
 We say  that a sequence $(\phi_n)\subset \mathcal{D}$ converges to $\phi\in \mathcal{D}$ if there is a compact set $K$ such that all $\phi_n$ have support in $K$ and for each integer $m\geq 0$, the sequence of derivatives $(\phi_n^{(m)})$ converges to $\phi^{(m)}$ uniformly on $K$, see e.g. \cite{A11}. The space of distributions on $\mathcal{D}$  is denoted by
  $\mathcal{D}'$. The space of  Denjoy integrable distributions  is defined by
\[\mathcal{A}_C:=\{f\in\mathcal{D}'\ |\ f=F' \text{ for } F\in \mathcal{B}_C\},\] where $$\mathcal{B}_C:=\{F:\mathbb{R}\rightarrow\mathbb{R}\ |\ F\in C^0(\overline{\mathbb{R}}), F(-\infty)=0\}.$$
We denote by   $C^0(\overline{\mathbb{R}})$ the continuous functions such that the limits  $\lim_{x\rightarrow\infty} F(x)$ and $\lim_{x\rightarrow-\infty}F(x)$ exist in $\mathbb{R}$. Setting $F(\pm\infty):=\lim_{x\rightarrow\pm\infty}F(x)$.

Denote $$BV:=\{g:\mathbb{R}\rightarrow\mathbb{R} \ | \ Vg<\infty\}$$ where $Vg:=\sup \Sigma|g(x_i)-g(y_i)|$ and the supremum is taken over all disjoint intervals $\{(x_i,y_i)\}$.



The convolution of  $g,f:B\subset \mathbb{R}\rightarrow\mathbb{R}$ is defined by
\begin{equation}\label{convolution}
g*f(x):=\int_B g(x-y)f(y)\,dy
\end{equation}
always that the integral (\ref{convolution}) exists in some sense.  It is well known that if $f,g\in L^1(\mathbb{R})$, then the convolution of $g$ and $f$ belongs to $L^1(\mathbb{R})$, \cite{A19}. Also, the convolution of $g\in BV$ and $f\in \mathcal{A}_c$  is defined with respect to the Henstock-Kurzweil distributional integral. On the other hand,  when $g\in L^1(\mathbb{R})$ and $f\in \mathcal{A}_c$, and knowing that $L^1(\mathbb{R})$ is dense in $\mathcal{A}_c$, their convolution is defined by
$$g*f(x):=\lim_{k\rightarrow\infty}g*f_k(x),$$
where $(f_k)\subset L^1(\mathbb{R})$ such that $||f_k-f||_A\rightarrow 0$, as $k\rightarrow \infty$.
Talvila \cite[Theorem 2.1, Theorem 3.4]{A22} proved the following result.
\begin{theorem} \label{the3} Let $(g,f)\in  BV\times\mathcal{A}_C$. Then
\begin{itemize}
\item[(i)]  $g*f(x)$ exists for each  $x\in\mathbb{R}$ and $g*f$ belongs to $C^0(\overline{\mathbb{R}})$;
\item[(ii)]  $f*g(x)=g*f(x)$;
\item[(iii)]  $||g*f||_\infty\leq ||f||_A||g||_{BV}$ where  $||g||_{BV}:=Vg+|g(-\infty)|$.
\end{itemize}
  Moreover, if $(g,f)\in L^1(\mathbb{R})\times \mathcal{A}_C$ and $h\in L^1(\mathbb{R})$, then
\begin{itemize}
\item[(i')] $g*f\in \mathcal{A}_c$;
\item[(ii')] $||g*f||_A\leq ||f||_A||g||_1$;
\item[(iii')]  $h*(g*f)(x)=(h*g)*f(x)$.
\end{itemize}
\end{theorem}

Note that,  any $f\in D_{HK}$ can be considered as an element in $\mathcal{A}_C$, because of its primitive $F\in C_0$ can be continuously  extended to a function in $C^0(\overline{\mathbb{R}})$.

\section{Riemann-Liouville fractional integral operator on $D_{HK}$}
In this section, we extend the Riemann-Liouville fractional integral operator over Henstock-Kurzweil integrable distributions and we prove fundamental properties, including the semigroup property.

In accordance with the convolution definition we set the following definition.
\begin{definition}\label{def n>1} Let $n\in\mathbb{R}^+\cup\{0\} $,  $f\in D_{HK}$ and \begin{equation}
\phi_n(u):=\left \{ \begin{matrix} u^{n-1} & \mbox{if \ } \mbox{} 0<u\leq b-a,
\\ 0 & \mbox{ }\mbox{ else.}\end{matrix}\right.\label{phi11}
\end{equation} The\textit{ Riemann-Liouville fractional integral operator of order $n$} is defined as
$$\mathcal{J}_a^nf(x):=\frac{1}{\Gamma(n)}\phi_n*f(x),$$
for $n\geq 1$ $$\phi_n*f(x):=\int_a^x(x-t)^{n-1}f(t)\,dt,$$
and for $0<n<1$ $$\phi_n*f(x):=\lim_{k\rightarrow\infty}\int_a^x(x-t)^{n-1}f_k(t)\,dt,$$whereby  $a\leq x\leq b$, $(f_k)\subset L^1[a,b]$ such that $||f_k-f||_A\rightarrow 0$, as $k\rightarrow \infty$. For $n=0$, we set $\mathcal{J}_a^0f:=I$, the identity operator.
\end{definition}

 \begin{remark} Note that if $n\geq 1$, then $\phi_n$ is increasing, non-negative and bounded on $[0,b-a]$. Thus, $\phi_n$  is a function of bounded variation on $\mathbb{R}$.  When $0<n<1$, the function $\phi_n$  belongs to  $L^1(\mathbb{R})$, but it is not a bounded variation function (is unbounded at $u=0$).  Observe that if  $f$ is in $L^1[a,b]$, then   $\mathcal{J}_a^nf(x)=J_a^nf(x)$,
since the distributional Henstock-Kurzweil integral contains the Lebesgue integral. Via the H\"older inequality, \cite{A24} and  \cite{A26}, it is easy to see that $\mathcal{J}_a^nf$ is a temperate distribution for any $f\in D_{HK}$ and $n\geq 0$.
 \end{remark}


In case $(a,b)=\mathbb{R}$,  the fractional integral over the real axis is defined analogously, see \cite{A20}. Note that for any $n>0,$ $\phi_n$ is not Lebesgue integrable  nor bounded on $\mathbb{R}$. Thus, convolution definition (Definition \ref{def n>1}) is not suitable for $f\in \mathcal{A}_c$ and $\phi_n$. On the other hand, $\phi_n\in L^1_{loc}(\mathbb{R})$, then $\phi_n$ defines a regular distribution,
$$\langle T_{\phi_n},\phi\rangle=\int_{-\infty}^\infty \phi_n(x)\phi(x)dx,$$ $\phi\in \mathcal{D}(\mathbb{R})$. The convolution of two distributions $S$ and $T$ is defined as
\begin{equation}\label{defi conv distri}
 \langle S*T,\phi\rangle=\langle S(\phi),\langle T(x),\phi(x-y)\rangle\rangle,
\end{equation}
see for example \cite{A8}. However, the inner pairing $\langle T(x),\phi(x-y)\rangle$ produces a function of $y$ which might not be
a test function.  Then the distributions must satisfy one
of the following conditions: (a) either $S$ or $T$ has bounded support,
(b) the supports of $S$ and $T$ are bounded on the same side. In the fractional integral case, these constraints force the distribution $f$ to have compact support. The Riemann-Liouville fractional derivative defined as convolution of generalized functions with the supports bounded on the same side is studied in \cite{A14}. Then, we will analyze the Riemann-Liouville fractional integral in the sense of the distributional Henstock-Kurzweil integral, according to Definition \ref{def n>1}.

Now, we will prove some fundamental properties of Riemann-Liouville fractional integrals.

\begin{theorem}\label{coro1} Let $n\in\mathbb{R}^+\cup\{0\} $, $f\in D_{HK}$ and $\mathcal{J}_a^nf(x)$ as in the Definition \ref{def n>1}.  Then,
\begin{enumerate}
\item [(i)]  $\mathcal{J}_a^n:D_{HK}\rightarrow D_{HK}$;
\item [(ii)] $\mathcal{J}_a^n$ is a bounded linear operator with respect to the Alexiewicz norm;

\item [(iii)] for $(f_k)\subset D_{HK}$ which converges in the Alexiewicz norm to $f$, we have that $(\mathcal{J}_a^nf_k)$ convergences in the Alexiewicz norm to $\mathcal{J}_a^nf$.

\item [(iv)] Moreover, if  $n\geq 1$, then   $$\mathcal{J}_a^nf(x)=\frac{1}{\Gamma(n)}\lim_{k\rightarrow\infty}\phi_n*f_k(x)$$ on $D_{HK}$ and as well on  $C[a,b]$, where $(f_k)\subset L^1[a,b]$ such that $||f_k-f||_A\rightarrow 0$, as $k\rightarrow\infty$.
\end{enumerate}
\end{theorem}
\begin{proof}
In the case $n=0$, claims $i),ii)$ and $iii)$ are trivial. Let $n\geq 1$ be fixed and $f\in D_{HK}$. By Theorem \ref{the3} we have that $\mathcal{J}_a^nf(x)$ exists for every $x\in [a,b]$, and  $\mathcal{J}_a^nf$ is  an element in $C_0$, thus $\mathcal{J}_a^n:D_{HK}\rightarrow C_0\subset D_{HK}$.  Now we will show that is a bounded linear operator with respect to the Alexiewicz norm. By Theorem \ref{the3} $(iii)$
\begin{eqnarray}
\left| \int_a^y\mathcal{J}_a^n f(x)dx \right|
&\leq&\int_a^y||\mathcal{J}_a^nf||_\infty dx\nonumber\\
&\leq&(y-a)\frac{1}{\Gamma(n)}||f||_A||\phi_n||_{BV}.\nonumber
\end{eqnarray}
Taking supreme when $y\in [a,b]$ we get
$$||\mathcal{J}_a^n f||_A\leq \frac{(b-a)}{\Gamma(n)}||f||_A||\phi_n||_{BV}.$$ Now let  $0<n<1$ be fixed. From Definition \ref{def n>1} and Theorem \ref{the3} we have $\phi_n*f\in \mathcal{A}_c$. Since $f$ and $\phi_n$ have compact support, we get that $\mathcal{J}_a^n:D_{HK}\rightarrow D_{HK}$.  It is clear that for any $n\in \mathbb{R}^+$, $\mathcal{J}_a^n$ is a linear operator  on $D_{HK}$,  because of the linearity of the integral.
By \cite[Theorem 1.32 ]{A18} we have $(iii)$.

 In the case $n\geq 1$, $\phi_n$  is of bounded variation
on $[0,b-a]$  and belongs to $L^1(\mathbb{R})$. Let us take a sequence $(f_k)$ in $L^1[a,b]$ such that $||f_k-f||_A\rightarrow 0$, as $k\rightarrow\infty$. By Theorem \ref{the3} $(iii)$ and $(i')$, we have
\[||\phi_n*f-\phi_n*f_k||_A\leq ||f-f_k||_A||\phi_n||_{1}\quad \text{ and } \] $$||\phi_n*f-\phi_n*f_k||_\infty\leq ||f-f_k||_A||\phi_n||_{BV}.$$Therefore,  $(iv)$ holds.
\end{proof}
 We will prove the semigroup property for the Riemann-Liouville fractional  integral operators.
\begin{theorem}\label{semigroup} Let $m,n\in\mathbb{R}^+\cup\{0\} $ and $f\in D_{HK}$. Then
\begin{itemize}
\item[(i)] $\mathcal{J}_a^m\mathcal{J}_a^nf=\mathcal{J}_a^{m+n}f$ in $D_{HK}$; Moreover, if $m\geq 1$ or $n\geq 1$, then the identity holds everywhere in $C[a,b]$;
\item[(ii)] $\mathcal{J}_a^m\mathcal{J}_a^nf=\mathcal{J}_a^n\mathcal{J}_a^mf$ in $D_{HK}$;
\item[(iii)] the set $\{\mathcal{J}_a^n:D_{HK}\rightarrow D_{HK}, n\geq 0\}$ forms a commutative semigroup with respect to concatenation. The identity operator $\mathcal{J}_a^0$ is the neutral element of this semigroup.
\end{itemize}
\end{theorem}

\begin{proof}
Let $n\in\mathbb{R}^+$, $f\in D_{HK}$. By Theorem \ref{coro1} $(iii)$  we have
$$\lim_{k\rightarrow\infty}\mathcal{J}_a^nf_k=\mathcal{J}_a^nf,$$
where $(f_k)\subset L^1[a,b]$ such that $||f-f_k||_A\rightarrow 0$, as $k\rightarrow \infty$.
For any $n,m\in\mathbb{R}^+$,
and for each $k\in\mathbb{N}$ we have
\begin{equation}
\mathcal{J}_a^m\mathcal{J}_a^nf_k(x)=\mathcal{J}_a^{m+n}f_k(x)\quad \text{a.e. on \ } [a,b], \label{3.2}
\end{equation}  and
\begin{equation}\mathcal{J}_a^m\mathcal{J}_a^nf_k(x)=\mathcal{J}_a^n\mathcal{J}_a^mf_k(x)\quad \text{a.e. on \ } [a,b],\label{3.3}
\end{equation} see \cite{A6}. Since  the composition of bounded operators is bounded, by (\ref{3.2}) we have
\[\mathcal{J}_a^m\mathcal{J}_a^nf=\mathcal{J}_a^m\lim_{k\rightarrow\infty}\mathcal{J}_a^nf_k=\lim_{k\rightarrow\infty}\mathcal{J}_a^m\mathcal{J}_a^nf_k=\lim_{k\rightarrow\infty}\mathcal{J}_a^{m+n}f_k= \mathcal{J}_a^{m+n}f\quad \text{in } D_{HK}.\]
If $n\geq 1$, then we have $\mathcal{J}_a^nf,\mathcal{J}_a^{m+n}f,\mathcal{J}_a^m\mathcal{J}_a^nf\in C[a,b] $. By Theorem \ref{coro1} $(iv)$
\[\mathcal{J}_a^m\mathcal{J}_a^nf(x)=\mathcal{J}_a^{m+n}f(x)\ \text{ on } [a,b] \]
and we obtain $(i).$
Claim $(ii)$ follows from (\ref{3.3}) and the continuity condition (Theorem \ref{coro1} $(iii)$)
\[\mathcal{J}_a^m\mathcal{J}_a^nf=\lim_{k\rightarrow\infty}\mathcal{J}_a^{m}\mathcal{J}_a^{n}f_k=\lim_{k\rightarrow\infty}\mathcal{J}_a^n\mathcal{J}_a^mf_k=\mathcal{J}_a^n\mathcal{J}_a^mf\quad \text{in } D_{HK}.\]
Finally $(iii)$ follows from the associative property (Theorem \ref{the3} $(ii')$).
\end{proof}

  Analogously, let us define the right-sided  \textit{Riemann-Liouville fractional integral of order $n$} on $D_{HK}$.

   \begin{definition}Let $n\in\mathbb{R}^+\cup\{0\} $,   $f\in D_{HK}$ and
 \begin{eqnarray}\psi_n(u) &:=& \left \{ \begin{matrix} (-u)^{n-1} & \mbox{if  }\  0< -u\leq b-a,
\\ 0 & \mbox{else. }\ \end{matrix}\right.\end{eqnarray}
The \textit{ right-side fractional integral of order $n$}  is    $$\mathcal{J}_{b_-}^nf(x):=\frac{1}{\Gamma(n)}\psi_n*f(x),$$
for $n\geq 1$, $$\psi_n*f(x):=\int^{b}_x(t-x)^{n-1}f(t)\,dt,
$$
and for $0<n<1$, $$\psi_n*f(x):=\lim_{k\rightarrow\infty}\int^{b}_x(t-x)^{n-1}f_k(t)\,dt,
$$
whereby $a\leq x\leq b$,  $(f_k)\subset L^1[a,b]$ such that $||f_k-f||_A\rightarrow 0$, as $k\rightarrow\infty$. For $n=0$, we set $\mathcal{J}_{b_-}^0:=I$, the identity operator.
\end{definition}


\begin{remark}It is clear that if $n\geq 1$, then $\psi_n$ is of bounded variation and belongs to $L^1(\mathbb{R})$. Thus, $\mathcal{J}_{b_-}^n:D_{HK}\rightarrow C_0\subset D_{HK}$, and for any $f\in D_{HK}$ $$\mathcal{J}_{b_-}^nf(x)=\lim_{k\rightarrow\infty}\mathcal{J}_{b_-}^nf_k(x),$$  in $D_{HK}$, where $(f_k)\subset L^1[a,b]$ such that $(f_k)$ converges in Alexiewicz norm to $f$. Analogously, we have   $||\mathcal{J}_{b_-}^nf||_\infty\leq ||f||_A||\psi_n||_{BV}/\Gamma(n)$, when $n\geq 1$. Moreover, we get that $\mathcal{J}^n_{b_-}$ is a bounded linear operator of $D_{HK}$ into $D_{HK}$, because of $||\mathcal{J}_{b_-}^nf||_A\leq ||f||_A||\psi_n||_{BV}(b-a)/\Gamma(n)$. If $0<n<1$, then $\psi_n\in L^1(\mathbb{R})$, and  $\mathcal{J}_{b_-}^n:D_{HK}\rightarrow D_{HK}$; by Theorem \ref{the3} we get that $\mathcal{J}_{b_-}^n$ is a bounded linear operator with respect to the Alexiewicz norm. Similarly, the semigroup property for the operators $\mathcal{J}_{b_-}^n$ is obtained.
\end{remark}

\section{Riemann-Liouville fractional differential operator on $D_{HK}$}

Now, we will extend the Riemann-Liouville differential operator  (Definition \ref{difoperator}) in distributional sense to get new fundamental properties between the fractional integral and differential operators. Moreover, we shall prove the fundamental theorem of fractional calculus on the space  $D_{HK}$.
\begin{definition}\label{def 4.1} Let $n\in\mathbb{R}^+\cup\{0\} $, $m:=\lceil n \rceil$ and $f\in D_{HK}$. The \textit{ Riemann-Liouville fractional differential operator of order $n$} is
$$\mathcal{D}_a^nf:=D^m\mathcal{J}_a^{m-n}f,$$
 where $D^m$ denotes the $m$-fold iterates of the distributional derivative. For $n=0$, we set $\mathcal{D}_a^0:=I$, the identity operator.
\end{definition}
\begin{remark} Observe that the operator $\mathcal{D}_a^n$ is well defined, since $\mathcal{J}_a^{m-n}f\in D_{HK}$, and  the distributional derivative of a distribution is a distribution, see \cite{A11}. Therefore, for any $n\in\mathbb{R}^+$, $$\mathcal{D}_a^n:D_{HK}\rightarrow \mathcal{D}'(a,b).$$
\end{remark}

\begin{remark} The Caputo derivative on $\mathbb{R}$ is given as
\begin{equation}\label{perpec1}^C D^nf(x):=\int_{-\infty}^x\frac{(x-t)^{m-n-1}}{\Gamma(m-n)}f^{(m)}(t)dt,
\end{equation} where $m:=\lceil n \rceil$. According to the convolution definition for distribution (\ref{defi conv distri}), if $f$ is any distribution with bounded support,
the Caputo and  Riemann-Liouville derivatives define the same functional in distributional sense, see  \cite{A14}; in particular it is valid when $f$ is induced by a function in $C^\infty[a,b]$.
In our study any distribution $f$ in $D_{HK}$  is considered. This means that $f$ might not be induced by a locally Lebesgue integrable function. 
\end{remark}

We will show that for any $n\in\mathbb{R}^+$ and $f\in D_{HK}$, the Riemann-Liouville integral operator can be written via
 the primitive  of $f$. This is a new property  even for Lebesgue integrable functions.
\begin{theorem}\label{prop2} Let $n\in\mathbb{R}^+\cup\{0\} $ and  $f\in D_{HK}$. Then, \begin{equation} \mathcal{J}_a^{n}f=D(\mathcal{J}_a^{n}F),
\label{7}
\end{equation}
where $F\in C_0$ and is the primitive of $f$.
In consequence, for $j\in \mathbb{N}$ and $\phi\in \mathcal{D}(a,b)$, then  \begin{equation}\langle D^{j}(\mathcal{J}_a^{n}f), \phi\rangle =(-1)^{j+1}\langle \mathcal{J}_a^{n}F, \phi^{(j+1)}\rangle.\label{8}
\end{equation}   Moreover, if $m=\lceil n \rceil$, then\begin{equation}\label{10}
\mathcal{D}^{n}_af= D^{m+1}\mathcal{J}^{m-n}F.
\end{equation}
For $0<n<1$, $\mathcal{D}^n_aF$ is a temperate distribution and
\begin{equation}\mathcal{D}^n_aF=\mathcal{J}_a^{1-n}f.\label{9}\end{equation}

\end{theorem}



\begin{proof}
The case $n=0$ is trivial. Let $n\in\mathbb{R}^+$ be fixed, $f\in D_{HK}$ and $\phi\in \mathcal{D}(a,b)$. By Definition \ref{def n>1} and Theorem \ref{coro1} $(iii)$ we have
\begin{eqnarray}
\langle \mathcal{J}_a^{n}f, \phi\rangle &:=& \int_a^b\mathcal{J}_a^{n}f(x)\phi(x)\,dx\nonumber\\
&=&\frac{1}{\Gamma(n)}\int_a^b\int_a^x(x-t)^{n-1}f(t)\,dt\,\phi(x)\,dx\nonumber\\&=&
\frac{1}{\Gamma(n)}\int_a^b\lim_{k\rightarrow\infty}\int_a^x(x-t)^{n-1}f_k(t)\phi(x)\,dt\,dx,\nonumber
\end{eqnarray}
where $(f_k)\subset L^1[a,b]$ such that $||f_k-f||_A\rightarrow\infty$, as $k\rightarrow\infty$. Since $ \mathcal{J}_a^{n}f_k$ converges in the Alexiewicz norm to $ \mathcal{J}_a^{n}f$ (Theorem \ref{coro1}), by Theorem \ref{convergence}, the Fubini Theorem 
 and integration by parts we have
$$\frac{1}{\Gamma(n)}\int_a^b\lim_{k\rightarrow\infty}\int_a^x(x-t)^{n-1}f_k(t)\phi(x)\,dt\,dx$$
\begin{eqnarray}&=&
\lim_{k\rightarrow\infty}\frac{1}{\Gamma(n)}\int_a^bf_k(t)\int_t^b(x-t)^{n-1}\phi(x)\,dx\,dt,\nonumber\\
&=&\lim_{k\rightarrow\infty}\frac{-1}{\Gamma(n)}\int_a^bf_k(t)\int_t^b \frac{(x-t)^n}{n}\phi'(x)\,dx\,dt\nonumber\\
&=&\lim_{k\rightarrow\infty}\frac{-1}{\Gamma(n)}\int_a^b\phi'(x)\int_a^xF_k(t)(x-t)^{n-1}\,dt\,dx\nonumber\\&=&-\lim_{k\rightarrow\infty}\int_a^b\phi'(x)\mathcal{J}_a^nF_k(x)\,dx.\nonumber
\end{eqnarray}
where $F_k(t)=\int_a^tf_k(r)\,dr$. Let $\varepsilon>0$, we have that  for  $k,k'$ large enough
\begin{equation*}
||F_k-F_{k'}||_\infty=||f_k-f_{k'}||_A<\varepsilon.
\end{equation*}
 Therefore, there exists $F\in C_0$ such that $\lim_{k\rightarrow\infty}F_k(x)=F(x)$ for $x\in [a,b]$ and  for a $k$ large enough we have $||F_k-F||_\infty<\varepsilon$. We will show that $||\mathcal{J}_a^nF_k-\mathcal{J}_a^nF||_A<\varepsilon$, for a $k$ large enough.

 Since $(F_k)$ and $F$ belong to $C_0$, we have
\begin{equation*}
||F_k-F||_A=\sup_{y\in[a,b]}\left|\int_a^yF_k(t)-F(t)dt\right|\leq\sup_{y\in[a,b]}\int_a^y||F_k-F||_\infty dt\leq\varepsilon(b-a),
\end{equation*}for a $k$ large enough, i.e. $F_k$ converges in the Alexiewicz norm to $F$. By the continuity property (Theorem \ref{coro1} $(iii)$) we get that $\mathcal{J}_a^nF_k$ converges to $\mathcal{J}_a^nF$ in the Alexiewicz norm as well.
Thus,   by  Theorem \ref{convergence}  we have
\begin{equation*}
\langle \mathcal{J}_a^{n}f, \phi\rangle = -\langle \mathcal{J}_a^{n}F, \phi'\rangle=\langle (\mathcal{J}_a^nF)',\phi\rangle,
\end{equation*}
and (\ref{7}) holds. From (\ref{7}) and the definition of derivative in the distributional sense we get (\ref{8}). By definition of fractional derivative we have $\mathcal{D}_a^nf=D^m\mathcal{J}_a^{m-n}f$, from where
 (\ref{10})  follows. Since $\mathcal{J}_a^{1-n}f$ is a temperate distribution and $\mathcal{D}_a^nF=D(\mathcal{J}_a^{1-n}F)$, we have (\ref{9}).
\end{proof}

\begin{corollary}\label{coro3} Let $n\in\mathbb{R}^+$, $m:=     \lceil n        \rceil$, $f\in D_{HK}$ and  $F\in C_0$  the primitive of $f$. Then
\begin{equation}\label{14}\mathcal{D}_a^nF=D^{m-1}\mathcal{J}_a^{m-n}f.
\end{equation}

\end{corollary}
\begin{proof}
The equality (\ref{14}) follows from expression (\ref{8}).
\end{proof}
The following example shows that, although $F$ is differentiable nowhere, its fractional derivative of any order is well defined. Moreover, the fractional integral and derivative of arbitrary order of $F'$ are always well defined in the distributional sense.
\begin{example}
\label{exam 1} Let  $n\in\mathbb{R}^+$ and $F\in C_0$ such that $F'(x)$ does not exists for any $x\in [a,b]$. Note that  $F\not\in ACG_*(I)$ for any $I\subset [a,b]$, hence $F'$ does not belong to $HK(I) $ nor $L^1(I)$. However,  $F'\in D_{HK}$ and $\int_a^xDF=F(x)$ for all $x\in [a,b]$.  By Theorem \ref{prop2} we have
 $$\mathcal{J}_a^{n}F'(\phi)=T_{\mathcal{J}_a^{n}F}'(\phi)=-\int_{-\infty}^\infty \mathcal{J}_a^nF(x)\phi'(x)dx\;\;   $$and $$\mathcal{D}_a^nF'(\phi)=T^{(m+1)}_{\mathcal{J}^{m-n}_a F}(\phi)=(-1)^{m+1}\int_{-\infty}^\infty \mathcal{J}_a^{m-n}F(x)\phi^{(m+1)}(x)dx\; ,$$
for all $\phi\in \mathcal{D}(a,b)$ where $\mathcal{J}_a^{n}F$ defines an absolutely continuous functions for any $n>0$ and $T^{(m)}$ denotes the m$-th$ distributional derivative of the distribution $T$. This means, the fractional integral and  derivative of any order   of $F'$ exist,  even though  $F'$ is not a Henstock-Kurzweil integrable function  and $F$ is differentiable nowhere, respectively. In particular, for $n\geq 1$, we have that $\mathcal{J}_a^{n}F'\in C^0(\overline{\mathbb{R}})$. On the other hand, by Definition \ref{def 4.1} and Corollary \ref{coro3} for all $\phi\in \mathcal{D}(a,b)$ $$\mathcal{D}_a^nF(\phi)=(-1)^{m}\int_{-\infty}^{\infty}\mathcal{J}_a^{m-n}F(x)\phi^{(m)}(x)dx=(-1)^{m-1}\mathcal{J}_a^{m-n}F'(\phi^{(m-1)}).$$
\end{example}



   \begin{lemma}\label{lemma4} For any $n\in\mathbb{R}^+ $, $f\in D_{HK}$ and $\phi\in\mathcal{D}(a,b)$ we have that  \begin{equation}
   \langle D^{j}(\mathcal{J}_a^{n}f), \phi\rangle=\langle D^{j}f,\mathcal{J}_{b_-}^n\phi\rangle,\label{operators r l}
\end{equation} where $j$ is any positive integer or zero.
   \end{lemma}

\begin{proof}
Let $n\in\mathbb{R}^+$, $\phi \in \mathcal{D}(a,b)$ and $f\in D_{HK}$. By Theorem \ref{coro1}, Theorem \ref{convergence} and the Fubini Theorem
\begin{eqnarray}
\langle \mathcal{J}_a^{n}f, \phi\rangle &:=&\lim_{k\rightarrow\infty}\frac{1}{\Gamma(n)}\int_a^bf_k(t)\int_t^b(x-t)^{n-1}\phi(x)\,dx\,dt,\nonumber\\
&=&\lim_{k\rightarrow\infty}\int_a^bf_k(t)\mathcal{J}_{b_-}^n\phi(t)\, dt.\label{rightoperator}
\end{eqnarray}
For each $\phi\in \mathcal{D}(a,b)$, $\mathcal{J}_{b_-}^n\phi\in AC[a,b]\subset BV[a,b]$. Since $(f_k)$ converges in the Alexiewicz norm to $f$, applying Theorem \ref{convergence} in (\ref{rightoperator}) we get
$$\langle \mathcal{J}_a^{n}f, \phi\rangle=\int_a^bf(t)\mathcal{J}_{b_-}^n\phi(t)dt=\langle f,\mathcal{J}_{b_-}^n\phi\rangle.$$ From here and by the definition of derivative in distributional sense, (\ref{operators r l}) holds.
\end{proof}

Applying integration by parts, semigroup property \cite{A20}, and FTC,  is easy to see that $J_{b_-}^n\phi\in \mathcal{D}(a,b)$. Since $\mathcal{J}_{b_-}^n\phi=J_{b_-}^n\phi$, (\ref{operators r l}) is well defined.   Now we will show a version of the integration by parts formula for $\mathcal{J}_a^n$ on $D_{HK}$.
   \begin{theorem} Let $n\in\mathbb{R}^+$, $f\in D_{HK}$ and $\phi\in BV[a,b]$. Then,
   \begin{equation}
   \int_a^b\phi(x)\mathcal{J}_a^nf(x)\,dx=\int_a^b f(t)\mathcal{J}_{b_-}^n\phi(t)\,dt.\label{22}
   \end{equation}
   \end{theorem}
   \begin{proof}
   Let $n\in\mathbb{R}^+$, $f\in D_{HK}$ and $\phi\in BV[a,b]$. Analogously, as in the proof of Lemma \ref{lemma4} we have
\begin{eqnarray}
 \int_a^b\phi(x)\mathcal{J}_a^nf(x)\,dx&=&  \int_a^b\phi(x)\lim_{k\rightarrow\infty}\frac{1}{\Gamma(n)}\int_a^x(x-t)^{n-1}f_k(t)\,dt\,dx\nonumber\\&=&\lim_{k\rightarrow\infty}\int_a^bf_k(t)\mathcal{J}_{b_-}^n\phi(t) \,dt,\nonumber\end{eqnarray}
where $(f_k)\subset L^1[a,b]$ and $||f_k-f||_A\rightarrow 0$, as $k\rightarrow\infty$.
 Since $\phi\in BV[a,b]$, we have that $\mathcal{J}_{b_-}^n\phi\in AC[a,b]$ and $ AC[a,b]\subset BV[a,b]$. 
By Theorem \ref{convergence} we have (\ref{22}).
   \end{proof}

\begin{theorem} Assume that $n_1,n_2\in\mathbb{R}^+\cup\{0\} $,  $g\in D_{HK}$ and $f=\mathcal{J}_a^{n_1+n_2}g$. Then $$\mathcal{D}_a^{n_1}\mathcal{D}_a^{n_2}f=\mathcal{D}_a^{n_1+n_2}f.$$
\end{theorem}
\begin{proof}
We proceed as in \cite[Theorem 2.13]{A6}.
\end{proof}

\begin{remark}
Note that $\mathcal{D}_a^{n_2}\mathcal{D}_a^{n_1}f=\mathcal{D}_a^{n_2+n_1}f=\mathcal{D}_a^{n_1+n_2}f=\mathcal{D}_a^{n_1}\mathcal{D}_a^{n_2}f$, for  $g\in D_{HK}$ and $f=\mathcal{J}_a^{n_1+n_2}g$.
\end{remark}

Now we will prove the fundamental theorem of calculus in the distributional sense, it means,  $\mathcal{D}_a^n$ inverts $\mathcal{J}_a^n$ on  $ D_{HK}$ and for any $n\in\mathbb{R}^+$.

\begin{theorem}\label{teo6} Let $n\in\mathbb{R}^+\cup\{0\} $. Then, for every $f\in D_{HK}$,
\begin{equation*}
\mathcal{D}_a^n\mathcal{J}_a^{n}f=f.
\end{equation*}
\end{theorem}
\begin{proof}
In the case $n=0$, the statement is trivial because $\mathcal{D}_a^n$ and $\mathcal{J}_a^{n}$ are both the identity operator. Assume that $n\in\mathbb{R}^+$ and let $m:=       \lceil  n\rceil$. Then, by the definition of $\mathcal{D}_a^n$, the semigroup property of the integral operators  and the fundamental theorem of calculus (Theorem \ref{FTCdis}),
$$\mathcal{D}_a^n\mathcal{J}_a^{n}f:=D^m\mathcal{J}_a^{m-n}\mathcal{J}_a^{n}f=D^m\mathcal{J}_a^{m}f=f.$$
\end{proof}

\begin{theorem} Let $f_1,f_2$ in $D_{HK}$, $n\in\mathbb{R}^+\cup\{0\} $ and $c_1, c_2\in\mathbb{R}$. Then,
$$\mathcal{D}^n(c_1f_1+c_2f_2)=c_1\mathcal{D}^nf_1+c_2\mathcal{D}^nf_2.$$
\end{theorem}
\begin{proof}
The linearity property follows from the linearity of the derivative.
\end{proof}

Now we will show some relations between Riemann-Liouville integrals and derivatives.

\begin{corollary}\label{cor2}
Let $n\in\mathbb{R}^+\cup\{0\} $. If there exists some $g\in D_{HK}$ such that $f=\mathcal{J}_a^{n}g$, then
\begin{equation*}
\mathcal{J}_a^{n}\mathcal{D}_a^{n}f=f.
\end{equation*}
\end{corollary}

\begin{proof}
It follows from definition of $f$ and Theorem \ref{teo6},
$$\mathcal{J}_a^{n}\mathcal{D}_a^{n}f=\mathcal{J}_a^{n}[\mathcal{D}_a^{n}\mathcal{J}_a^{n}g]=\mathcal{J}_a^{n}g=f.$$
\end{proof}



\begin{corollary}Let $0<n<1$. Assume that $f$ is such that $\mathcal{J}_a^{1-n}f\in C_0$. Then, \begin{equation}\mathcal{J}_a^n\mathcal{D}_a^nf=f.\label{19}\end{equation}
In particular, if $f\in C_0$, then (\ref{19}) holds as well.
\end{corollary}

\begin{proof}
Assume that $\mathcal{J}_a^{1-n}f\in C_0$. Since $C_0$ is isometrically isomorphic to $D_{HK}$ (Theorem \ref{isomorphic}) there exists $\phi\in D_{HK}$ such that
\begin{equation}\label{20}\mathcal{J}_a^{1-n}f:=\mathcal{J}_a^1\phi.\end{equation}
Applying the operator $\mathcal{D}_a^{1-n}$ in (\ref{20}) by  Theorem \ref{semigroup} and Theorem \ref{teo6},
\begin{equation*}\label{21}f=\mathcal{J}_a^{n}\phi.\end{equation*}
It follows from Corollary \ref{cor2}. Now let us consider $f\in C_0$, then $\mathcal{J}_a^{1-n}f\in C_0$ it follows in an analogous way.
\end{proof}

We have that the differential operator $D$ is the inverse operator  $\mathcal{ J}_a^1$. This means, $$D\mathcal{J}_a^{1}f=f,$$
for any $f\in D_{HK}$.  Analogously, for $n\in\mathbb{R}^+$ and $f\in D_{HK}
$ $$\mathcal{D}_a^n\mathcal{J}_a^{n}f=f.$$
Now, let us consider  $F\in C[a,b]$. 
  In general, $\int_a^xDF(t)dt\neq F(x)$ because of the appearance of the constant $F(a)$.
We will show that in general, $\mathcal{J}_a^n$ does not inverts $D_a^n$. Even more, we provide an explicit expression.



\begin{definition}\label{defCn}
Let $n\in\mathbb{N}$. It is said that $F\in \mathcal{C}^{n}[a,b]$ if and only if $F^{(n-1)}\in C[a,b]$ in the classical sense, it means, $\left(\frac{d}{dx}\right)^{n-1} F(x)= F^{(n-1)}(x)$.
\end{definition}

It is clear that $\mathcal{C}^1[a,b]=C[a,b]$. Note that if $F^{(n-1)}\in C[a,b]$, then for any $\phi\in \mathcal{D}(a,b)$ we have $$\langle F^{(n-1)},\phi\rangle=(-1)^{n-1}\langle F,\phi^{(n-1)}\rangle.$$  Moreover, if $F\in \mathcal{C}^n[a,b]$, then $F, F^{(1)},...,F^{(n-1)}\in C[a,b]$ and $F^{(k-1)}$ is the primitive of $F^{(k)}$, therefore $$F^{(k-1)}(x)=\int_a^xF^{(k)}(t\,)dt+F^{(k-1)}(a),$$ with another notation
$$D^{k-1}F=F^{(k-1)},$$
where $k=1,2,...,n-1.$ For $k=n$ we have  $D^nF=\phi$ for some $\phi\in D_{HK}$.

\begin{lemma}\label{lemma 4.9} The space $\mathcal{C}^n[a,b]$ consists of those and only those functions $f(x)$, which are represented in the form
$$f(x)=\frac{1}{(n-1)!}\int_a^x(x-t)^{n-1}\phi(t)\,dt+\sum_{k=0}^{n-1}c_k(x-a)^k,$$
where $\phi\in D_{HK}$, $c_k$ being arbitrary constants.
\end{lemma}

\begin{proof} The proof follows from Definition \ref{defCn},  characterization of $C[a,b]$ in terms of distributional integral and  Definition \ref{def n>1}. Moreover, $\phi= D^{n}f$  and $c_k=f^{(k)}(a)/k!=D^kf(a)/k!$.
\end{proof}

Let $n$ be a positive number. We define    $\lfloor n \rfloor$ as the greatest integer less than or equal to $n$.

\begin{theorem} Let $n\in\mathbb{R}^+$ and $m:=	\lfloor n \rfloor+1$. Assume that $f$ is such that $\mathcal{J}_a^{m-n}f\in \mathcal{C}^m[a,b]$. Then, $$\mathcal{J}_a^n\mathcal{D}_a^nf(x)=f(x)-\sum_{k=0}^{m-1}\frac{(x-a)^{n-k-1}}{\Gamma(n-k)}D^{m-k-1}\mathcal{J}_a^{m-n}f(a).$$
 In particular, if $\ 0<n<1$ we have
$$\mathcal{J}_a^n\mathcal{D}_a^nf(x)=f(x)-\frac{(x-a)^{n-1}}{\Gamma(n)}\mathcal{J}_a^{1-n}f(a).$$
\end{theorem}

   \begin{proof}
   Since $\mathcal{J}_a^{m-n}f\in C^m$, by Theorem \ref{isomorphic} there exits $\phi\in D_{HK}$ such that
   $$D^{m-1}\mathcal{J}_a^{m-n}f=\mathcal{J}_a^1\phi+D^{m-1}\mathcal{J}_a^{m-n}f(a).$$   By Lemma \ref{lemma 4.9}  we have that
\begin{equation}\label{23} \mathcal{J}_a^{m-n}f(x)=\mathcal{J}_a^m\phi(x)+\sum_{k=0}^{m-1}\frac{(x-a)^{k}}{k!}D^k\mathcal{J}_a^{m-n}f(a).
\end{equation}
By definition of $\mathcal{D}_a^n$, the expression (\ref{23}) and Theorem \ref{semigroup}
\begin{eqnarray}
\mathcal{J}_a^n\mathcal{D}_a^nf(x)&:=&\mathcal{J}_a^n\mathcal{D}^m\mathcal{J}_a^{m-n}f(x)\nonumber\\&=&\mathcal{J}_a^n\mathcal{D}^m\left[\mathcal{J}_a^m\phi+\sum_{k=0}^{m-1}\frac{(x-a)^{k}}{k!}D^k\mathcal{J}_a^{m-n}f(a)\right]\nonumber\\
&=&\mathcal{J}_a^n\mathcal{D}^m\mathcal{J}_a^m\phi+\sum_{k=0}^{m-1}\frac{\mathcal{J}_a^nD^m(\cdot-a)^{k}(x)}{k!}D^k\mathcal{J}_a^{m-n}f(a)\nonumber\\&=&\mathcal{J}_a^n\phi.\label{17}
\end{eqnarray}
 Now applying the operator $\mathcal{D}_a^{m-n}$ to both sides of (\ref{23}) and by Theorem \ref{semigroup} we have that
   \begin{eqnarray*}
   f(x)=\mathcal{D}_a^{m-n}\mathcal{J}_a^m\phi(x)+\sum_{k=0}^{m-1}\frac{\mathcal{D}_a^{m-n}[(\cdot-a)^k](x)}{k!}D^k\mathcal{J}_a^{m-n}f(a).
   \end{eqnarray*}
   By definition of derivative, Theorem \ref{semigroup} and Example 2.4 in \cite{A6},
     \begin{eqnarray}\label{18}
   f(x)=\mathcal{J}_a^n\phi(x)+\sum_{k=0}^{m-1}\frac{(x-a)^{k+n-m}}{\Gamma(k+n-m+1)}D^k\mathcal{J}_a^{m-n}f(a).
   \end{eqnarray}
So, we substitute $k$ by $m-k-1$ in (\ref{18})  and by the expression (\ref{17}), we obtain the result.
   \end{proof}

Following the idea that the primitives of an element in $D_{HK}$ differ by a constant, we will generalize the following results, \cite[Lemma 2.1]{A20}. Let $0<n<1$ and $F\in AC[a,b]$. Then $J_a^{1-n}F \in AC[a,b]$ and
   $$J_a^{1-n}F(x)=\frac{F(a)(x-a)^{-n+1}}{\Gamma(-n+2)}+J_a^{2-n}F'(x).$$
   Furthermore, under the same assumptions of \cite[Lemma 2.12]{A6} we have  $$D_a^nF(x)=\frac{F(a)}{\Gamma(1-n)(x-a)^n}+J_a^{1-n}F'(x).$$
Now we give the corresponding generalizations in the distributional sense.

   \begin{corollary}\label{corollary 4} Let  $0<n<1$ and $F\in C[a,b]$. Then $\mathcal{J}_a^{1-n}F \in C[a,b]$ and
   $$\mathcal{J}_a^{1-n}F(x)=\frac{F(a)(x-a)^{-n+1}}{\Gamma(-n+2)}+\mathcal{J}_a^{2-n}f(x)$$
  for   $f\in D_{HK}$ such that $f=DF$. Moreover,
   $$\mathcal{D}_a^nF(x)=\frac{F(a)}{\Gamma(1-n)(x-a)^n}+\mathcal{J}_a^{1-n} f(x).$$
   \end{corollary}

   \begin{proof}
  Let $F\in C[a,b]$. By Theorem \ref{isomorphic} and Theorem \ref{FTCdis} we have that  $$F(x)=\int_a^xf(z)dz+F(a)=\int_a^xDF+F(a),$$ for some $f\in D_{HK}$ such that $f=DF$. Let $0<n<1$ be fixed. Then
   \begin{eqnarray}
   \mathcal{J}_a^{1-n}F(x)&:=&\frac{1}{\Gamma(1-n)}\int_a^x(x-t)^{-n}F(t)\,dt\nonumber\\
   &=&\frac{1}{\Gamma(1-n)}\int_a^x(x-t)^{-n}\left(F(a)+\int_a^tf(z)dz\right)dt.\label{F}
\end{eqnarray}
Let us consider $f\in D_{HK}$,  $h(u)=u^{-n} $ if $ 0< u\leq b-a$ and 0 else, and let us define their definite integrals
as \[F_0(t)=\int_a^tf,\quad t\geq a\quad \text{ and }\quad H_0(x)=\int_0^xh,\quad x\geq0.\] We emphasize with the subindex zero that the definite integrals $F_0$ and $H_0$ are in $C_0$ and $AC[a,b]\cap C_0$,  respectively. It is clear that $f$ can be considered as an element of $\mathcal{A}_c$ and $h$ belongs to $L^1(\mathbb{R})$. By \cite[Theorem 4.4]{A22} we have that $H_0*f\in C_0$ and $H_0*f(x)=h*F_0(x)$ for all $x\in\mathbb{R}$.
Moreover,
\begin{eqnarray}
\int_\mathbb{R}h(x-t)F_0(t)\,dt&=&\int_a^xh(x-t)F_0(t)\,dt\nonumber\\&=&\int_\mathbb{R}H_0(x-t)f(t)\,dt\nonumber\\&=&\int_a^xH_0(x-t)f(t)\,dt,\label{convolu}
\end{eqnarray}
for $a\leq x\leq b$. On the other hand,
\begin{eqnarray}
H_0*f(x)=\int_a^x\frac{(x-t)^{-n+1}}{-n+1}f(t)\,dt.\label{16}
\end{eqnarray}
 By expressions (\ref{F}), (\ref{convolu}), (\ref{16}) and properties of gamma function we have that
$$  \frac{1}{\Gamma(1-n)}\int_a^x(x-t)^{-n}\left(F(a)+\int_a^tf(z)dz\right)dt$$
   \begin{eqnarray}
  &=&\frac{F(a)}{\Gamma(2-n)}(x-a)^{1-n}+ \frac{1}{\Gamma(1-n)}h*F_0(x)\nonumber\\
   &=&\frac{F(a)}{\Gamma(2-n)}(x-a)^{1-n}+ \frac{1}{\Gamma(1-n)}H_0*f(x)\nonumber\\
   &=&\frac{F(a)}{\Gamma(2-n)}(x-a)^{1-n}+\mathcal{J}_a^{2-n}f(x).\nonumber
\end{eqnarray}
It is clear that $\mathcal{J}_a^{1-n}F\in C[a,b]$. On the other hand,
\begin{eqnarray}
\mathcal{D}_a^n F(x)&:=&D\mathcal{J}^{1-n}F(x)\nonumber\\&=&D\left[\frac{F(a)(x-a)^{-n+1}}{\Gamma(-n+2)}+\mathcal{J}_a^{2-n}f(x)\right]\nonumber\\
&=&\frac{F(a)}{\Gamma(1-n)(x-a)^n}+D\mathcal{J}_a^{2-n}f(x).\nonumber
\end{eqnarray}
 By Theorem \ref{prop2}, Theorem \ref{semigroup} and Theorem \ref{teo6}  we have
$$D\mathcal{J}_a^{2-n}f(x)=DD\mathcal{J}_a^{2-n}F_0(x)=D\mathcal{J}_a^{1-n}F_0(x)=\mathcal{J}_a^{1-n}f(x).$$
   \end{proof}

   \begin{example} \label{exam 2} Let $f:[0,1]\rightarrow\mathbb{R}$ defined as
$$f(t)=\left \{ \begin{matrix} (-1)^{k+1}2^kk^{-1} & \mbox{if \ } \mbox{} t\in[c_{k-1},c_k)
\\ 0 & \mbox{ if\ }\mbox{ }   t=1,\end{matrix}\right.$$
where $c_i=1-2^{-i}$, $i=0,1,2,...$
It is well known that $f\in HK[0,1]\setminus L^1[0,1]$, see \cite{A2}. Therefore, $F(t):=\int_0^tf$ belongs to $ACG_*[0,1]\setminus AC[0,1]$, see e.g., \cite{A10}.  Since $F$ is in $C_0$, we have that $\mathcal{J}_0^{1-n}F$ belongs to $AC[0,1]$ for $0<n<1$, hence $\mathcal{J}_0^{1-n}F$ is differentiable a.e. Taking $n=1/2$ and  applying Corollary   \ref{corollary 4} we can calculate  $\mathcal{J}_0^{1/2}f$,  even though $f\not\in L^1[0,1]$,
\begin{eqnarray}
\mathcal{J}_0^{1/2}f(x)&=&\mathcal{D}_0^{1/2}F(x)\nonumber\\
&=& \frac{1}{\sqrt{\pi}}\left[\sum_{k=1}^{n-1}\frac{(-1)^{k+1}}{k}+2c_{n-1}\frac{(-1)^{n+1}}{n}\right] x^{-1/2}+\frac{(-1)^{n+1}2^{n+1}}{n\pi}x^{1/2},\nonumber
\end{eqnarray}
when $x\in (c_{n-1}, c_n)$, in case $x=c_i$ for  $i\in\mathbb{N}$ we have that $\mathcal{D}_0^{1/2}F(x)$ does not exists. This example illustrates that $f$ does not need to be Lebesgue integrable to achieve its fractional integral pointwise a.e. Moreover,  according to Definition \ref{def 4.1} and equality (\ref{7}) from Theorem \ref{prop2} the operators
$\mathcal{J}_0^{1/2}f$ and $\mathcal{D}_0^{1/2}F$ are distributions given by the same  regular distribution, this means that for all $\phi\in\mathcal{D}(a,b)$
$$\mathcal{J}_0^{1/2}f(\phi)=-\int_{-\infty}^{\infty}\mathcal{J}_0^{1/2}F(x)\phi'(x)dx=\int_{-\infty}^{\infty}\mathcal{J}_0^{1/2}f(x)\phi(x)dx=\mathcal{D}_0^{1/2}F(\phi),$$
and  $$\mathcal{D}_0^{1/2}f(\phi)=\int_{-\infty}^{\infty}\mathcal{J}_0^{1/2}F(x)\phi''(x)dx=-\int_{-\infty}^{\infty}\mathcal{J}_0^{1/2}f(x)\phi'(x)dx. $$  \end{example}


\section{Some applications}

The integral equation
\begin{eqnarray}\label{Abel equa}
\frac{1}{\Gamma(n)}\int_a^x(x-t)^{n-1}\varphi(t\,)dt=f(x),
\end{eqnarray}
where $0<n<1$,  $f$ is a known function and $\varphi$ is the unknown function, is called the Abel integral equation. Assume this equation is considered on a compact interval $[a,b]$. We will show a generalization of the solution of (\ref{Abel equa}).

\begin{theorem} The Abel integral equation (\ref{Abel equa}) is solvable in $D_{HK}$ if and only if $\mathcal{J}_a^{1-n}f\in C[a,b]$ and $\mathcal{J}_a^{1-n}f(a)=0$.
\end{theorem}

\begin{proof}
\textit{Necessity}. Let $0<n<1$ be fixed. Assume that the Abel integral equation (\ref{Abel equa}) has a solution in $D_{HK}$. It means, given $f\in D_{HK}$ there exists $\varphi\in D_{HK}$ such that $$\mathcal{J}_a^n\varphi(x)=f(x).$$
Applying $\mathcal{J}_a^{1-n}$, by Theorem \ref{semigroup} we get
$$\mathcal{J}_a^1\varphi(x)=\int_a^x\varphi(t)\,dt=\mathcal{J}_a^{1-n}f(x).$$By hypothesis $\varphi\in D_{HK}$, thus $\mathcal{J}_a^{1-n}f\in C_0$. Since the Abel integral equation has a solution, it must be $\varphi=\mathcal{D}_a^nf$.

\textit{Sufficiency}.  Assume that there exists $f\in D_{HK}$ such that  $\mathcal{J}_a^{1-n}f\in C_0$. Then there exists $\varphi_1\in D_{HK}$ such that
\begin{equation*}\label{f}\mathcal{J}_a^{1-n}f(x)=\int_a^x\varphi_1(t)\,dt.
\end{equation*}
We will prove that $\varphi_1$ is solution of (\ref{Abel equa}). Applying $\mathcal{D}_a^{1-n}$ by Theorem \ref{teo6} and Theorem \ref{semigroup} we have
$$f=\mathcal{J}_a^n\varphi_1.$$
\end{proof}
\begin{remark}It is possible to prove this result using the associative property of the convolution, change of variable and the isomorphism between $D_{HK}$ and $C_0$.
\end{remark}

The Fourier transform of a distribution $T$ is defined as
$$\langle \widehat{T},\phi\rangle=\langle T,\widehat{\phi}\rangle,$$
where $\phi\in\mathcal{ D}$, the set of test functions,  and the Fourier transform of $ g$ at $x$ is $$\widehat{g}(x)=\int_{-\infty}^\infty g(t)\exp(-2\pi itx)dt.$$
Thus, we obtained the following properties of the Fourier transform for the Riemann-Liouville integral and derivative.
\begin{proposition} Let $n\in\mathbb{R}^+\cup\{0\} $, $F\in C_0$ and $f\in D_{HK}$ such that $DF=f$. Then
\begin{itemize}
\item[(i)] $$\langle\widehat{\mathcal{J}_a^nf},\phi\rangle=\langle\eta \widehat{\mathcal{J}_a^nF},\phi\rangle,$$ where $\eta=2\pi i s$ and $\phi\in\mathcal{ D}(a,b)$.
\item[(ii)] For $m-1\leq n< m$, $$ \langle\widehat{\mathcal{D}_a^nf},\phi\rangle=\langle\eta^{m}\widehat{\mathcal{J}_a^{m-n}f},\phi\rangle=\langle\eta^{m+1}\widehat{\mathcal{J}_a^{m-n}F},\phi\rangle.$$
\item[(iii)] Let $k\in\mathbb{N}$. Then $$\langle(D^k\widehat{\mathcal{J}_a^n f}),\phi\rangle=\langle\widehat{(-\eta)^k\mathcal{J}_a^nf,\phi}\rangle,$$
where $D^{k}$ denotes the $k-$folds iterates of the derivative in distributional sense.
\end{itemize}
\end{proposition}
\begin{proof}By (\ref{7}) from Theorem \ref{prop2}
\begin{eqnarray}
\langle\widehat{\mathcal{J}_a^nf},\phi\rangle&=&\langle\mathcal{J}_a^nf,\widehat{\phi}\rangle\nonumber\\
&=&-\langle\mathcal{J}_a^nF,(\widehat{\phi})\ '\rangle\nonumber\\
&=&-\langle\mathcal{J}_a^nF,\widehat{\psi}\rangle\nonumber\\
&=&\langle\widehat{\mathcal{J}_a^nF},\eta \phi\rangle\nonumber\\
&=&\langle\eta\widehat{\mathcal{J}_a^nF}, \phi\rangle,\label{last}
\end{eqnarray}
where $\psi=-2\pi s\phi(s)$.
Analogously, by definition of distributional derivative and (\ref{last}) we obtain $(ii)$. Let $k\in\mathbb{N}$, it is easy to see
$$\langle D(\widehat{\mathcal{J}_a^n f})
,\phi\rangle=\langle\widehat{-\eta\mathcal{J}_a^nf},\phi\rangle.$$
Applying the same argument $k-1$ times we get the result.
\end{proof}
Also, if $T$ is a distribution and  $\psi$ is a test function, then   the convolution of $\psi*T$ is defined as
$$\langle\psi*T,\phi\rangle=\langle T,\psi^-*\phi\rangle,$$where $\psi
^-(x)=\psi(-x)$ and $\phi\in\mathcal{ D}(a,b)$. 
 Additionally, the convolution theorem holds in distributional sense, see \cite{A7}:  if  $T$ is a temperate distribution and $\psi$ a distribution with compact support, then
$$\widehat{\psi*T}=\widehat{\psi}\cdot\widehat{T},$$
where the right side is the product of a distribution and a  function in $C^\infty(\mathbb{R})$, which defines a temperate distribution. From the previous proposition we obtained the following relations.   Note that $\mathcal{J}_a^nf$ is a temperate distribution when $f\in D_{HK}$ and $n\geq 0$, see \cite{A24}, \cite{A26}.

\begin{proposition}Let $n\in\mathbb{R}^+\cup\{0\}$, $F\in C_0$, $f\in D_{HK}$ such that $DF=f$,  and $\psi, \phi\in \mathcal{D}(a,b)$.
\begin{itemize}
\item[(i)]  Then  $$\langle\widehat{\psi*\mathcal{J}_a^nf},\phi\rangle=\langle\eta\widehat{ \mathcal{J}_a^nF} \cdot\widehat{\psi},\phi\rangle.$$
\item[(ii)] If $\mathcal{D}_a^nf$ is a temperate distribution, $$\langle\widehat{\psi*\mathcal{D}_a^nf},\phi\rangle=\langle\widehat{\eta^{m+1}\mathcal{J}_a^{m-n}F\cdot\widehat{\psi},\phi}\rangle
,$$ where $m-1<n\leq m$.
\end{itemize}
\end{proposition}

 \section*{Conclusions}
In this work, we present an extension of the Riemann-Liouville fractional derivative defined over the space of Henstock-Kurzweil integrable distributions. Thus, we obtain a generalization of classical results and new relations between the fractional derivative and integral operators. These results enable to achieve expressions for the fractional integral (of arbitrary order) of any integrable distribution $f$, even though $f$ is not induced by a  Lebesgue integrable function; and for the fractional derivative of any order of  $f'$, though $f$ is continuous and differentiable nowhere, see Example \ref{exam 1} and Example \ref{exam 2}.  Hence, an advantage of this extension is that the fractional derivative of any order $n>0$ (integer or non-integer) is always well defined over  $D_{HK}$, which contains the spaces of Lebesgue, Henstock-Kurzweil and improper integrable functions. Thus, it is not necessary to constrain the fractional derivative $D^n$ over  $AC^n$. We believe that our results might contribute to modeling real-world problems. 

\section*{Perspectives}

In the classic fractional calculus  theory there is a version of Leibniz' formula. This means, it is assumed that  the functions $f,g$ are analytic functions in order to obtain  an explicit expression of the fractional derivative of the product $fg$, see e.g. \cite[Theorem 2.18]{A6}. In the distribution theory the product of a distribution $f\in\mathcal{D}'(a,b)$ and $g\in C_c^\infty(a,b)$ is defined as
$$fg(\phi):=f(g\phi),$$
for any $\phi\in C_c^\infty(a,b)$. In particular, if $f\in D_{HK}$ and $g\in BV[a,b]$, then the product belongs to $D_{HK}$ and $$\int_a^bfg=F(b)g(b)-\int_a^bFdg,$$
where $F\in C_0$ and $DF=f$, see \cite{A24}. In this paper we  extended the integral and differential concepts of arbitrary positive order in the distributional sense. Thus a natural question is:  Under what conditions is it possible to get an explicit expression of $\mathcal{D}^n[fg]$, where $n\in\mathbb{R}^+$, $f\in D_{HK}$ and $g\in BV[a,b]$?

Here we established fundamental results for fractional calculus in the sense of the distributional Henstock-Kurzweil integral.  On the other hand, the Riemann-Liouville fractional derivative seems to be the most suitable according to theoretical and applied studies, see  \cite{At5}. Nevertheless, many numerical approximations of the fractional derivative were made considering the Caputo derivative and the Lebesgue integral. Thus,  the future possible research is rich and  has several directions, for example,  differential equations, generalized differential equations,   mathematical modeling and  numerical approximation, where integration techniques play an important role, see e.g., \cite{At1},  \cite{At4},  \cite{At5}, \cite{At6}, \cite{A3}, \cite{A9}, \cite{A12}, \cite{Mo1},  \cite{A21},  \cite{Ta1}, \cite{Ye1} and \cite{A27}, among others. 




\medskip
Received December 2019; revised February 2020.
\medskip

\end{document}